\documentclass[11pt]{amsart}

\usepackage{amscd}
\usepackage{amsmath, amssymb, verbatim}
\usepackage{amsfonts}
\newcommand{\de}{\partial}

\newcommand{\dbar}{\overline{\partial}}
\newcommand{\ddt}{\frac{\partial}{\partial t}}

\newcommand{\ddbar}{\sqrt{-1} \partial \overline{\partial}}
\newcommand{\Ric}{\mathrm{Ric}}
\newcommand{\ov}[1]{\overline{#1}}

\newcommand{\tr}[2]{\mathrm{tr}_{#1}{#2}}

\newcommand{\vp}{\varphi}

\newcommand{\PSH}{\mathrm{PSH}}
\newcommand{\SRF}{\mathrm{SRF}}

\newcommand{\om}{\omega}
\newcommand{\Om}{\Omega}
\renewcommand{\bar}{\overline}
\newcommand{\tta}{\widetilde}

\newcommand{\Kod}{\textrm{Kod}}

\renewcommand{\leq}{\leqslant}
\renewcommand{\geq}{\geqslant}
\renewcommand{\le}{\leqslant}
\renewcommand{\ge}{\geqslant}

\numberwithin{equation}{section}

\begin{document}
\newtheorem{claim}{Claim}
\newtheorem{theorem}{Theorem}[section]
\newtheorem{lemma}[theorem]{Lemma}
\newtheorem{corollary}[theorem]{Corollary}
\newtheorem{proposition}[theorem]{Proposition}
\newtheorem{question}[theorem]{Question}
\newtheorem{conjecture}[theorem]{Conjecture}
\newtheorem{problem}[theorem]{Problem}

\theoremstyle{definition}
\newtheorem{remark}[theorem]{Remark}
\newtheorem{example}[theorem]{Example}

\title[Convergence of scalar curvature]{Convergence of scalar curvature of K\"ahler-Ricci flow on manifolds of positive Kodaira dimension}

\author{Wangjian Jian}

\address{School of Mathematical Sciences, Peking University, Yiheyuan Road 5,Beijing, P.R.China, 100871}

\email{1401110008@pku.edu.cn}

\pagestyle{headings}

\begin{abstract}
In this paper, we consider K\"ahler-Ricci flow on n-dimensional K\"ahler manifold with semi-ample canonical line bundle and $0<m:=\Kod(X)<n$.  Such manifolds admit a Calabi-Yau fibration over its canonical model.  We prove that the scalar curvature of the K\"ahler metrics along the normalized K\"ahler-Ricci flow converge to $-m$ outside the singular set of this fibration.
\end{abstract}

\maketitle

\section{Introduction}
\label{section:introduction}

Let us first recall the set up of Song-Tian \cite{ST, ST2, ST3} where our result will apply.  Let $(X^n,\omega_0)$ be a compact K\"ahler manifold with canonical line bundle $K_X$ being semi-ample and $0<m:=\Kod(X)<n$.  Therefore the canonical ring $R(X,K_X)$ is finitely generated, and so the pluricanonical system $|\ell K_X|$ for sufficiently large $\ell \in \mathbb{Z^+}$ induces a holomorphic map
\begin{equation}\label{canonical map}
f:X \to B \subset \mathbb{C}\mathbb{P}^N:=\mathbb{P}H^0(X,K_X^{\otimes \ell}),
\end{equation}
where $B$ is the canonical model of $X$.  We have $dimB=m$.

Let $S'$ be the singular set of $B$ together with the set of critical values of $f$, and we define $S=f^{-1}(S') \subset X$.

Now let $\omega(t)$ be the smooth global solution of the normalized K\"ahler-Ricci flow
\begin{equation}\label{KRF1}
\frac{\de \omega}{\de t}=-Ric(\omega)-\omega,~\omega|_{t=0}=\omega_0.
\end{equation} 
It's well-known \cite{TiZha, Ts} that the flow has a global solution on $X\times[0,\infty)$.  It's shown by Song-Tian \cite{ST,ST2} that $\omega(t)$ collapses nonsingular Calabi-Yau fibers and the flow converges weakly to a generalized K\"ahler-Einstein metric $\omega_B$ on its canonical model $B$, with $\omega_B$ is smooth and satisfies the generalized Einstein equation on $B\backslash S'$
\begin{equation}\label{tKEE on canonical model}
Ric(\omega_B)=-\omega_B+\omega_{\mathrm{WP}},
\end{equation}
where $\omega_{\mathrm{WP}}$ is the Weil-Petersson metric induced by the Calabi-Yau fibration $f$.  They also proved the $C^0$-convergence on the potential level and in the case when $X$ is an elliptic surface the $C_{loc}^{1,\alpha}$-convergence of potentials on $X\backslash S$ for any $\alpha < 1$.  In \cite{ST3}, Song-Tian showed that the scalar curvature is uniformly bounded on $X\times [0,\infty)$ along the normalized flow. The case when $X$ is of general type is given by Z.Zhang in \cite{Z1}.  The case for conical K\"ahler-Ricci flow is given by G.Ewards in \cite{Ed}.

In \cite{FZ}, Fong-Zhang proved the $C^{1,\alpha}$-convergence of potentials when $X$ is a global submersion over $B$ and showed the Gromov-Hausdorff convergence in the special case.  In \cite{TWY} Tosatti-Weinkove-Yang improved the estimate and showed that the metric $\omega(t)$ converges to $f^*\omega_B$ in the $C^0$ local-topology on $X\backslash S$. Moreover, Tosatti-Weinkove-Yang \cite{TWY} proved that the restricted metric $\omega(t)|_{X_y}$ converges (up to scalings) in the $C^0$-topology to the unique Ricci flat metric on the fibre $X_y$ for any regular value $y$; this result is improved to be smooth convergence by Tosatti-Zhang in \cite{ToZh}.  Also see Tosatti's note \cite{To15} for clearer and more unified discussions.

In fact, Tosatti-Weinkove-Yang \cite{TWY} obtained in their proof that $\|\om(t)-\tta\om(t)\|_{\om(t)}\to 0$ as $t\to \infty$ on $X\backslash S$, where $\tta\om(t)=e^{-t}\om_\SRF+(1-e^{-t})\om_B$ (see Section 2 for definition of $\om_\SRF$).  This enable us to prove that $|\tr{\om(t)}{\om_B}-m|+\left|\|\om_B\|_{\om(t)}^2-m\right|\to 0$ as $t\to \infty$ on $X\backslash S$, which then enable us to improve the estimate of scalar curvature on $X\backslash S$, following the argument of Song-Tian \cite{ST3}.  In this paper, we prove that the scalar curvature $R(t)$ converges on the regular part $X\backslash S$.
\begin{theorem}\label{main theorem}
Let $(X,\omega_0)$ be given as above, let $\omega(t)$ be the smooth global solution of the normalized K\"ahler-Ricci flow \eqref{KRF1}. Then we have
\begin{equation}\label{scc1}
\lim_{t\to \infty}R(t)=-m,~on~X\backslash S\times [0,\infty).
\end{equation}
In particular, if $S=\emptyset$, then $f$ is a holomorphic submersion and we have
\begin{equation}\label{scc2}
\left|R(t)+m\right|\le Ce^{-\eta t},~on~X\times [0,\infty),
\end{equation}
for some constants $\eta,C>0$ depending on $(X,\omega_0)$.
\end{theorem}
After rescaling time and space simultaneously, we have the following immediately corollary from Theorem \ref{main theorem} of the unnormalized K\"ahler-Ricci flow.
\begin{corollary}\label{corollary}
Let $(X,\omega_0)$ be given as above, let $\omega(t)$ be the smooth global solution of the unnormalized K\"ahler-Ricci flow
\begin{equation}\label{unscc3}
\frac{\de \omega}{\de t}=-Ric(\omega),~\omega|_{t=0}=\omega_0.
\end{equation}
Then we have
\begin{equation}\label{unscc4}
\lim_{t\to \infty}(1+t)R(t)=-m,~on~X\backslash S\times [0,\infty).
\end{equation}
In particular, if $S=\emptyset$, then $f$ is a holomorphic submersion and we have
\begin{equation}
\left|(1+t)R(t)+m\right|\le\frac{C}{(1+t)^\eta},~on~X\times [0,\infty),
\end{equation}
for some constants $\eta,C>0$ depending on $(X,\omega_0)$.
\end{corollary}

Note that in Theorem \ref{main theorem}, the limiting behavior of scalar curvature on the singular set $S$ is unknown.  A recent result of the author and two other authors \cite{JSS} says that: If the canonical bundle $K_X$ is semi-ample,  then for any K\"ahler class $[\omega]$ on $X$, there exists $\delta_{X, [\omega]}>0$ such that for any $0<\delta<\delta_{X, [\omega]}$, there exists a unique cscK metric in the K\"ahler class $ [K_X]+\delta[\omega]$.  Hence we can propose the following conjecture.
\begin{conjecture}
Let $X$ be an $n$-dimensional K\"ahler manifold with nef canonical bundle $K_X$ and positive Kodaira dimension. Then for any initial K\"ahler metric $\om_0$, the solution $\om(t)$ of the normalized K\"ahler-Ricci flow 
$$\ddt{\om} = - Ric(\om) - \om, ~ \om(0)=\om_0$$ converges in Gromov-Hausdorff topology to $\om_{B}$ and the scalar curvature $R(t)$ converges to  $-\textnormal{Kod}(X)$ in $C^0(X)$, where $\textnormal{Kod}(X)$ is the Kodaira dimension of $X$.  
\end{conjecture}

In general, it is natural to ask if the following holds  for the maximal solution of the unnormalized K\"ahler-Ricci flow on $X\times [0, T)$, where $X$ is a K\"ahler manifold and $T>0$ is the maximal existence time.

\begin{enumerate}
\item If $T<\infty$, then there exists $C>0$ such that $$-C\leq R(t)\leq C(T-t)^{-1}. $$

\item If $T=\infty$, then there exists $C>0$ such that $$|R(t) | \leq C (1+t) ^{-1}. $$
\end{enumerate}

In \cite{SeT}, the answer to the first question is affirmative due to Perelman for the K\"ahler-Ricci flow  on Fano manifolds with finite time extinction. In \cite{Z2}, it is shown that if the K\"ahler-Ricci flow develops finite time singularity, the scalar curvature blows up at most of rate $(T-t)^{-2}$ if $X$ is projective and if the initial K\"ahler class lies in $H^2(X, \mathbb{Q})$.

{\bf Acknowledgements.} The author would like to thank his advisor Gang Tian for leading him to study K\"ahler-Ricci flow, constant encouragement and support.  The author would like to thank Jian Song for helpful discussions.  The author also would like to thank Yalong Shi and Dongyi Wei for helpful discussions.  This work was carried out while the author was visiting Jian Song at the Department of Mathematics of Rutgers University, supported by the China Scholarship Council (File No.201706010022).  The author would like to thank the China Scholarship Council for supporting this visit.  The author also would like to thank Jian Song and the Department of Mathematics of Rutgers University for hospitality and support.

\section{Preliminary for the K\"ahler Ricci-flow}
\label{Preliminary}
In this section let us recall some known results that we need in our proof.

From \eqref{canonical map}, we have $f^*\mathcal{O}(1)=K_X^{\otimes \ell}$, hence if we let $\chi=\frac{1}{\ell}\omega_{\mathrm{FS}}$ on $\mathbb{P}H^0(X,K_X^{\otimes \ell})$, we have that $f^*\chi$ (later, denoted by $\chi$) is a smooth semi-positive representative of $-c_1(X)$.  Here, $\omega_{\mathrm{FS}}$ denotes the Fubini-Study metric.  Also, we denote by $\chi$ the restriction of $\chi$ to $B\backslash S'$. 

Given a K\"ahler metric $\omega_0$ on $X$, since $X_y:=f^{-1}(y)$ are Calabi-Yau for $y\in B\backslash S'$, there exists a unique smooth function $\rho_y$ on $X_y$ with $\int_{X_y}\rho_y\omega_0^{n-m}=0$, and such that $\omega_0|_{X_y}+\ddbar\rho_y=:\omega_y$ is the unique Ricci-flat K\"ahler metric on $X_y$.  Moreover, $\rho_y$ depends smoothly on $y$, and so define a global smooth function on $X\backslash S$.  We define
$$\omega_{\SRF}=\omega_0+\ddbar\rho,$$
which is a closed real $(1,1)$-form on $X\backslash S$, restricts to a Ricci-flat K\"ahler metric on all fibers $X_y$ of $y\in B\backslash S'$.

Let $\Omega$ be the smooth volume form on $X$ with
\begin{equation}\label{background volume form}
\ddbar\log\Omega=\chi,~\int_X\Omega=\binom{n}{m}\int_X\omega_0^{n-m}\wedge\chi^m.
\end{equation}
Define a function $F$ on $X\backslash S$ by
\begin{equation}
F:=\frac{\Omega}{\binom{n}{m}\chi^m\wedge\omega_{\SRF}^{n-m}},
\end{equation}
then $F$ is constant along the fiber $X_y$, $y\in B\backslash S'$, so it descends to a smooth function on $B\backslash S^\prime$.  Then \cite{ST2} showed that the Monge-Amp\'ere equation
\begin{equation}\label{MAE on canonical model}
(\chi +\ddbar v)^m=Fe^v\chi^m,
\end{equation}
has a unique solution $v\in \PSH(\chi)\cap C^0(B)\cap C^\infty(B\backslash S')$.  Define
$$\om_B=\chi+\ddbar v ,$$
which is a smooth K\"ahler metric on $B\backslash S'$, satisfies the twisted K\"ahler-Einstein equation
$$\Ric(\om_B)=-\om_B+\om_{\mathrm{WP}},$$
where $\om_{\mathrm{WP}}$ is the smooth Weil-Petersson form on $B\backslash S'$.

Now let $\om=\om(t)$ be the solution of the normalized K\"ahler-Ricci flow
\begin{equation}\label{nKRF}
\ddt\om=-\Ric(\om)-\om,~\om(0)=\om_0,
\end{equation}
which exists for all time.  Define the reference metrics
$$\hat{\om}(t)=e^{-t}\om_0+(1-e^{-t})\chi,$$
which are K\"ahler for all $t\ge 0$, and we can write $\om(t)=\hat{\om}(t)+\ddbar\vp(t)$, and $\vp(0)=0$, then the K\"ahler-Ricci flow \eqref{nKRF} is equivalent to the parabolic Monge-Amp\'ere equation
\begin{equation}\label{scalar MAE}
\ddt\vp=\mathrm{log}\frac{e^{(n-m)t}\left(\hat{\om}(t)+\ddbar\vp(t)\right)^n}{\Om}-\vp,~\vp(0)=0.
\end{equation}

From now on, we always set $K=f^{-1}(K')$， where $K'\subset B\backslash S'$ is a compact subset.  Then we can choose some open subset $U'\subset\subset B\backslash S'$ such that $K'\subset U'$.  Set $U=f^{-1}(U')$, then $K\subset \subset U\subset \subset X\backslash S$.  Also, we denote by $h(t)$ some positive decreasing function on $[0,\infty)$ which tends to zero as $t\to \infty$.

Now we have the following lemmas.  See \cite{TWY,To15} for unified discussions (and also \cite{FZ,ST2,ST3}).
\begin{lemma}\label{basic}
There exists some constant $C=C(K)$ and $h(t)$ depending on the domain $K$, such that
\begin{enumerate}
\item [(1)]~$C^{-1}\hat{\om}(t)\le\om(t)\le C\hat{\om}(t)$, on $K\times[0,\infty).$ 
\item [(2)]\quad $|\vp-v|+|\dot{\vp}+\vp-v|\le h(t)$,\quad on $K\times[0,\infty).$
\item [(3)]There exists a uniform $C_0>0$ such that
$$|R|\le C_0,~on~X\times[0,\infty).$$
\item [(4)]~$\tr{\om(t)}{\om_B}-m\le h(t),~on~K\times[0,\infty).$
\item [(5)]Especially, if $S=\emptyset$, then (1)-(4) hold with $K$ replaced by $X$ and $h(t)$ replaced by $Ce^{-\eta t}$ for some constants $\eta,C>0$ depending on $(X,\om_0)$.
\end{enumerate}
\end{lemma}

\begin{lemma}\label{Schwarz Lemma}
Along the normalized K\"ahler Ricci-flow, we have on $X\backslash S\times[0,\infty)$
\begin{equation}\label{evolution of u}
\left(\ddt-\Delta\right)(\dot{\vp}+\vp-v)=\tr{\om(t)}{\om_B}-m.
\end{equation}
and there exists some $C=C(K)>0$ such that
\begin{equation}\label{Schwarz inequality}
\left(\ddt-\Delta\right)\tr{\om(t)}{\om_B}\le C,~on~K\times[0,\infty).
\end{equation}
Especially, when $S=\emptyset$, then \eqref{evolution of u}, \eqref{Schwarz inequality} holds on $X\times[0,\infty)$ with $C$ depending on $(X,\om_0)$.
\end{lemma}

Next we define on $X\backslash S$ the reference metrics
$$\tta{\om}(t)=e^{-t}\om_{\SRF}+(1-e^{-t})\om_B.$$
Then we have the following theorem due to \cite{TWY} (in the proof).
\begin{theorem}\label{C^0 convergence}
There exists $h(t)$ depending on the domain $K$ such that
\begin{equation}\label{C^0 c1}
\|\om(t)-\tta\om(t)\|_{C^0(K,\om(t))}\le h(t).
\end{equation}
Especially, when $S=\emptyset$, then
\begin{equation}\label{C^0 c2}
\|\om(t)-\tta\om(t)\|_{C^0(X,\om(t))}\le Ce^{-\eta t}.
\end{equation}
for some constants $\eta,C>0$ depending on $(X,\om_0)$.
\end{theorem}

We also need the following lemma to choose local coordinates on the regular part, see e.g. Lemma 5.6 of \cite{To15}.
\begin{lemma}\label{local coordinates}
Let $f:X^n\to Y^m$ be a holomorphic submersion between complex manifolds.  Then given any point $x\in X$ we can find an open set $U\ni x$ and local holomorphic coordinates $(z_1,\dots,z_n)$ on $U$ and $(y_1,\dots,y_m)$ on $f(U)$ such that in these coordinates the map $f$ is given by $(z_1,\dots,z_n)\mapsto(z_1,\dots,z_m)$, i.e., $y_1=z_1,\dots,y_m=z_m.$
\end{lemma}

We can apply Lemma \ref{local coordinates} to a point $x\in X\backslash S,y=f(x)\in B\backslash S'$ to choose local coordinates, and we may call such coordinates ``local product coordinates''.

\section{Convergence of the trace and norm of $\om_B$ along the flow}
From now on, we denote by $T_0=\tr{\om(t)}{\om_B}$.

In this section, we use Theorem $\ref{C^0 convergence}$ to prove $|T_0-m|+\left|\|\om_B\|_{\om(t)}^2-m\right|\to 0$ as $t\to \infty$ on $X\backslash S$.  As before, we use $h(t),h_1(t),\dots$ to denote positive decreasing functions on $[0,+\infty)$ which tends to zero as $t\to\infty$.

First, we have the following basic estimate.
\begin{lemma}\label{C^0 estimate}
For any point $x\in U$ with local product coordinates given by Lemma \ref{local coordinates} around $x$ and $y=f(x)$, say $(z_1,\dots,z_n)$ around $x$ and $(y_1,\dots,y_m)$ around $y$.  Suppose on such coordinate neighborhood $\om(t)$ is given by 
$$\om(t)=\sum_{i,j=1}^{n}g(t)_{i\ov j}dz_i\wedge d\ov z_{ j},$$
then there exists some constant $C$ depending on the domain such that: for $1\le\alpha,\beta\le m$, $m+1\le i,j\le n,$
\begin{equation}\label{C^0 lower}
\left|g(t)_{\alpha\ov\beta}\right|\le C,~\left|g(t)_{\alpha\ov j}\right|\le Ce^{-\frac{t}{2}},~\left|g(t)_{i\ov j}\right|\le Ce^{-t}.
\end{equation}
\begin{equation}\label{C^0 upper}
\left|g(t)^{\alpha\ov\beta}\right|\le C,~\left|g(t)^{\alpha\ov j}\right|\le Ce^{\frac{t}{2}},~\left|g(t)^{i\ov j}\right|\le Ce^t.
\end{equation}
at $x$.  In particular, when $S=\emptyset$, \eqref{C^0 lower} and \eqref{C^0 upper} hold on $X\times[0,\infty)$ with $C$ depending on $(X,\om_0)$. 
\end{lemma}
\begin{proof}
We define on such coordinate neighborhood (contained in $U$) the local metrics
\begin{equation}\label{local et coord}
\om_E(t)=\om^{(m)}+e^{-t}\om^{(n-m)},
\end{equation}
where $\om^{(m)}$ and $\om^{(n-m)}$ denotes the standard Euclidean metrics on the two factors of $\mathbb{C}^n=\mathbb{C}^m\times\mathbb{C}^{n-m}$.  Thanks to Lemma \ref{basic}, we can find constant $C$ depending on the domain $K$ such that 
\begin{equation}\label{local et equivalence}
C^{-1}\om_E(t)\le \om(t)\le C\om_E(t),
\end{equation}
Denote
$$\om_B=\sqrt{-1}\sum_{\alpha,\beta=1}^{m}(g_B)_{\alpha\bar{\beta}}dz_{\alpha}\wedge d\bar{z}_{\beta},~\om_\SRF=\sqrt{-1}\sum_{i,j=1}^{n}(g_\SRF)_{i\bar{j}}dz_i\wedge d\bar{z}_j,$$
Hence using \eqref{C^0 c1} of Theorem $\ref{C^0 convergence}$ we have
\[
\begin{split}
&C\geq \|\om(t)-\tta\om(t)\|_{\om_E(t)}^2\\
&\geq \sum_{\alpha,\beta=1}^{m}\left|g(t)_{\alpha\bar{\beta}}-(1-e^{-t})(g_B)_{\alpha\bar{\beta}}-e^{-t}(g_\SRF)_{\alpha\bar{\beta}}\right|^2\\
&+\sum_{\alpha=1}^{m}\sum_{j=m+1}^{n}e^t\left|g(t)_{\alpha\bar{j}}-e^{-t}(g_\SRF)_{\alpha\bar{j}}\right|^2+\sum_{i,j=m+1}^ne^{2t}\left|g(t)_{i\bar{j}}-e^{-t}(g_\SRF)_{i\bar{j}}\right|^2,
\end{split}
\]
on our coordinate neighborhood.  Then we can apply the trivial inequality $|a-b|^2\geq \frac{1}{2}|a|^2-|b|^2$ to each term to conclude \eqref{C^0 lower}. Next, from Lemma \ref{basic}, we have the first and third estimates of \eqref{C^0 upper}, and the second estimate then follows from Cauchy-Schwarz inequality.
\end{proof}

\begin{proposition}\label{base convergence}
There exists $h(t)$ depending on $K$ such that
\begin{equation}\label{base trace convergence}
\left|T_0-m\right|\leq h(t),~on~K\times[0,\infty).
\end{equation}
\begin{equation}\label{base norm convergence}
\left|\|\om_B\|_{\om(t)}^2-m\right|\leq h(t),~on~K\times[0,\infty).
\end{equation}
In particular, if $S=\emptyset$, then we have
\begin{equation}\label{base convergence2}
\left|T_0-m\right|+\left|\|\om_B\|_{\om(t)}^2-m\right|\leq Ce^{-\eta t},~on ~X\times[0,\infty),
\end{equation}
where $\eta,C>0$ are constants depending on $(X,\om_0).$ 
\end{proposition}
\begin{proof}
Applying Theorem $\ref{C^0 convergence}$ to $\bar{U}$ there exists some $h_1(t)$ depending on the domain such that
\begin{equation}\label{local C^0 convergence}
\|\om(t)-\tta\om(t)\|_{C^0(\bar{U},\om(t))}^2\leq h_1(t).
\end{equation}
Now given any $x_0\in K$, we choose local product coordinate like Lemma $\ref{C^0 estimate}$, say $U_0\subset U$ around $x_0$.  WLOG, we may assume
$$U_0=B^{(m)}(1)\times B^{(n-m)}(1)\subset \mathbb{C}^n,~f(U_0)=B^{(m)}(1)\subset \mathbb{C}^m,~x_0=(0,0),$$
where $B^{(m)}(1)$ and $B^{(n-m)}(1)$ denotes Euclidean unit balls in $\mathbb{C}^m$ and $\mathbb{C}^{n-m}$, respectively.  The map $f$ is given by $f(z_1,\dots,z_n)=(z_1,\dots,z_m)$.  Fix a time $t$, we define the transformation
\[
\begin{split}
F_t:\quad &U_1:=B^{(m)}(1)\times B^{(n-m)}(e^{-\frac{t}{2}})\to U_0\\
&F_t(w_1,\dots,w_n)=(w_1,\dots,w_m,e^{\frac{t}{2}}w_{m+1},\dots,e^{\frac{t}{2}}w_n),
\end{split}
\]
Then consider on $U_1$ the metrics
$$\om_1(t)=F_t^*\om(t),\qquad \tta \om_1(t)=F_t^*\tta\om(t).$$
We immediately have from $\eqref{local C^0 convergence}$ that
\begin{equation}\label{local C^0 convergence after normal}
\|\om_1(t)-\tta\om_1(t)\|_{\om_1(t)}^2(x_0)=\|\om(t)-\tta\om(t)\|_{\om(t)}^2(x_0)\leq h_1(t).
\end{equation}
Denote on $U_1$
$$\om_1(t)=\sqrt{-1}\sum_{i,j=1}^{n}h(t)_{i\bar{j}}dw_i\wedge d\bar{w}_j,~\tta\om_1(t)=\sqrt{-1}\sum_{i,j=1}^{n}\tta h(t)_{i\bar{j}}dw_i\wedge d\bar{w}_j,$$
then by definition
\[
\begin{split}
&\om_1(t)=F_t^*\left[\sqrt{-1}\sum_{i,j=1}^{n}g(t)_{i\bar{j}}dz_i\wedge d\bar{z}_j\right]\\
=&\sqrt{-1}\sum_{\alpha,\beta=1}^{m}g(t)_{\alpha\bar{\beta}}dw_\alpha\wedge d\bar{w}_\beta+2Re\left[\sqrt{-1}\sum_{\alpha=1}^{m}\sum_{j=m+1}^{n}e^{\frac{t}{2}}g(t)_{\alpha\bar{j}}dw_\alpha\wedge d\bar{w}_j\right]\\
&+\sqrt{-1}\sum_{i,j=m+1}^{n}e^tg(t)_{i\bar{j}}dw_i\wedge d\bar{w}_j,\\
\end{split}
\]
hence we obtain: for $1\leq\alpha,\beta\leq m,m+1\leq i,j\leq n$
\begin{equation}\label{local tran of metric}
\left\{
       \begin{aligned}
       &h(t)_{\alpha\bar{\beta}}=g(t)_{\alpha\bar{\beta}},~ h(t)^{\alpha\bar{\beta}}=g(t)^{\alpha\bar{\beta}},\\
       &h(t)_{\alpha\bar{j}}=e^{\frac{t}{2}}g(t)_{\alpha\bar{j}},~ h(t)^{\alpha\bar{j}}=e^{-\frac{t}{2}}g(t)^{\alpha\bar{j}},\\
       &h(t)_{i\bar{j}}=e^tg(t)_{i\bar{j}},~h(t)^{i\bar{j}}=e^{-t}g(t)^{i\bar{j}}.\\
       \end{aligned}\right.
\end{equation}
Then we apply Lemma $\ref{C^0 estimate}$ to conclude there exists some constant $C=C(K)>0$ such that
\begin{equation}\label{local C^0 estimate after normal}
\sum_{i,j=1}^{n}\left(\left|h(t)_{i\bar{j}}\right|+\left|h(t)^{i\bar{j}}\right|\right)(x_0)\leq C.
\end{equation}
Similarly, we have: for $1\leq \alpha,\beta\leq m,m+1\leq i,j\leq n$
\begin{equation}\label{local tran of bg metric}
\left\{
       \begin{aligned}
       &\tta h(t)_{\alpha\bar{\beta}}=(1-e^{-t})(g_B)_{\alpha\bar{\beta}}+e^{-t}(g_\SRF)_{\alpha\bar{\beta}},\\
       &\tta h(t)_{\alpha\bar{j}}=e^{-\frac{t}{2}}(g_\SRF)_{\alpha\bar{j}},\\
       &\tta h(t)_{i\bar{j}}=(g_\SRF)_{i\bar{j}}.\\
       \end{aligned}
\right.
\end{equation}
Now at $x_0$ we define an $n\times n$ matrix
\begin{gather*}
      A=(a_{i\bar{j}})_{n\times n}=
      \begin{pmatrix}
      (g_B)_{\alpha\bar{\beta}}(x_0)& 0\\ 0 & (g_\SRF)_{i\bar{j}}(x_0)
      \end{pmatrix}_{1\leq\alpha,\beta\leq m,m+1\leq i,j\leq n}.
\end{gather*}
We claim that 
\begin{equation}\label{C^0 close to A}
\|\om_1(t)-A\|_{\om_E}^2(x_0)\leq h_2(t).
\end{equation}
for some $h_2(t)$ depending on the domain, where $\om_E$ is the standard Euclidean metric on $U_1$.  To see this, we use \eqref{local C^0 convergence after normal} and \eqref{local C^0 estimate after normal} to obtain
\begin{equation}\label{detailed local C0}
      \begin{split}
      Ch_1(t)
      &\geq\|\om_1(t)-\tta\om_1(t)\|_{\om_E}^2(x_0)\\
      &=\sum_{i,j=1}^{n}\left|h(t)_{i\bar{j}}-\tta h(t)_{i\bar{j}}\right|^2(x_0).\\
      \end{split}
\end{equation}
But at $x_0$
\[
\begin{split}
\|\om_1(t)-A\|_{\om_E}^2=
&\sum_{\alpha,\beta=1}^{m}\left|h(t)_{\alpha\bar{\beta}}-a_{\alpha\bar{\beta}}\right|^2+2\sum_{\alpha=1}^{m}\sum_{j=m+1}^{n}\left|h(t)_{\alpha\bar{j}}\right|^2\\
&+\sum_{i,j=m+1}^{n}\left|h(t)_{i\bar{j}}-a_{i\bar{j}}\right|^2.\\
\end{split}
\]
Then we can use \eqref{local tran of bg metric} and \eqref{detailed local C0} to estimate three terms on the RHS of the above equality.  Indeed, for $1\leq\alpha,\beta\leq m$, we have at $x_0$
\[
\begin{split}
Ch_1(t)
&\geq \left|h(t)_{\alpha\bar{\beta}}-\tta h(t)_{\alpha\bar{\beta}}\right|^2=\left|h(t)_{\alpha\bar{\beta}}-(1-e^{-t})(g_B)_{\alpha\bar{\beta}}-e^{-t}(g_\SRF)_{\alpha\bar{\beta}}\right|^2\\
&\geq \frac{1}{2}\left|h(t)_{\alpha\bar{\beta}}-a_{\alpha\bar{\beta}}\right|^2-e^{-2t}\left|(g_B)_{\alpha\bar{\beta}}-(g_\SRF)_{\alpha\bar{\beta}}\right|^2,\\
\end{split}
\]
which gives
$$\sum_{\alpha,\beta=1}^{m}\left|h(t)_{\alpha\bar{\beta}}-a_{\alpha\bar{\beta}}\right|^2\leq 2Ch_1(t)+2e^{-2t}\left|(g_B)_{\alpha\bar{\beta}}-(g_\SRF)_{\alpha\bar{\beta}}\right|^2\leq h_3(t),$$
Next, for $1\leq\alpha\leq m,m+1\leq j\leq n$, we have
\[
\begin{split}
Ch_1(t)
&\geq \left|h(t)_{\alpha\bar{j}}-\tta h(t)_{\alpha\bar{j}}\right|^2=\left|h(t)_{\alpha\bar{j}}-e^{-\frac{t}{2}}(g_\SRF)_{\alpha\bar{j}}\right|^2\\
&\geq\frac{1}{2}\left|h(t)_{\alpha\bar{j}}\right|^2-e^{-t}\left|(g_\SRF)_{\alpha\bar{j}}\right|^2,\\
\end{split}
\]
which gives
$$\sum_{\alpha=1}^{m}\sum_{j=m+1}^{n}\left|h(t)_{\alpha\bar{j}}\right|^2\leq 2Ch_1(t)+2e^{-t}\left|(g_\SRF)_{\alpha\bar{j}}\right|^2\leq h_4(t),$$
Finally, for $m+1\leq i,j\leq n$, we have
$$Ch_1(t)\geq \sum_{i,j=m+1}^{n}\left|h(t)_{i\bar{j}}-(g_\SRF)_{i\bar{j}}\right|^2=\sum_{i,j=m+1}^{n}\left|h(t)_{i\bar{j}}-a_{i\bar{j}}\right|^2,$$
Combine the above three estimates we obtain \eqref{C^0 close to A} with
$$h_2(t)=Ch_1(t)+h_3(t)+h_4(t).$$

Now at $x_0$ we can use \eqref{local C^0 estimate after normal} and \eqref{C^0 close to A} to get
\[
\begin{split}
&\left|\det\om_1(t)-\det A\right|\\
&=\left|\sum_{(j_1,\dots,j_n)}(-1)^{\sigma(j_1,\dots,j_n)}(h(t)_{1\bar{j}_1}\cdots h(t)_{n\bar{j}_n}-a_{1\bar{j}_1}\cdots a_{n\bar{j}_n})\right|\\
&\leq \sum_{(j_1,\dots,j_n)}\left|h(t)_{1\bar{j}_1}\cdots h(t)_{(n-1)\bar{j}_{n-1}}\left(h(t)_{n\bar{j}_n}-a_{n\bar{j}_n}\right)\right|+\\
&\quad\dots+\sum_{(j_1,\dots,j_n)}\left|\left(h(t)_{1\bar{j}_1}-a_{1\bar{j}_1}\right)h(t)_{2\bar{j}_2}\cdots h(t)_{n\bar{j}_n}\right|\\
&\leq C\sum_{(j_1,\dots,j_n)}\left(\left|h(t)_{1\bar{j}_1}-a_{1\bar{j}_1}\right|+\cdots+\left|h(t)_{n\bar{j}_n}-a_{n\bar{j}_n}\right|\right)\\
&\leq Ch_2(t)^{\frac{1}{2}}.\\
\end{split}
\]
Hence we obtain
\begin{equation}\label{C^0 close of det}
\left|\det\om_1(t)(x_0)-\det A\right|\leq h_5(t).
\end{equation}
But $\det A\in[A_0,A_1]$ for some positive constants $A_0,A_1$ depending on the domain, independent of $t$ (after $t$ is large), so we can choose a large time $T\geq 1$ such that for all $t\geq T$, we have
\begin{equation}\label{C^0 bound of det}
\det\om_1(t)(x_0)\in \left[\frac{1}{2}A_0,2A_1\right],
\end{equation}
Set $A^{-1}=\left(a^{i\bar{j}}\right)$.  Now we use \eqref{local C^0 estimate after normal}, \eqref{C^0 close to A}, \eqref{C^0 close of det}, \eqref{C^0 bound of det} to estimate at $x_0$
\[
\begin{split}
&\left|h(t)^{1\bar{1}}-a^{1\bar{1}}\right|\\
&=\left|\sum_{(j_2,\cdots,j_n)}\left(\frac{(-1)^{\sigma(j_2,\cdots,j_n)}h(t)_{2\bar{j}_2}\cdots h(t)_{n\bar{j}_n}}{\det\om_1(t)}-\frac{(-1)^{\sigma(j_2,\cdots,j_n)}a_{2\bar{j}_2}\cdots a_{n\bar{j}_n}}{\det A}\right)\right|\\
&\leq \frac{1}{\det\om_1(t)}\sum_{(j_2,\cdots,j_n)}\left|h(t)_{2\bar{j}_2}\cdots h(t)_{n\bar{j}_n}-a_{2\bar{j}_2}\cdots a_{n\bar{j}_n}\right|\\
&\quad+\sum_{(j_2,\cdots,j_n)}\left|a_{2\bar{j}_2}\cdots a_{n\bar{j}_n}\right|\left|\frac{1}{\det\om_1(t)}-\frac{1}{\det A}\right|\\
&\leq\frac{1}{2}Ch_2(t)^{\frac{1}{2}}+C\frac{h_5(t)}{\frac{1}{2}A_0\cdot A_0}\leq h_6(t),\\
\end{split}
\]
Similar argument holds for all $1\leq i,j\leq n$, hence we obtain
\begin{equation}\label{upper C^0 close}
\sum_{i,j=1}^{n}\left|h(t)^{i\bar{j}}(x_0)-a^{i\bar{j}}\right|\leq h_6(t).
\end{equation}
By the special form of the matrix $A$, we have
$$\sum_{\alpha,\beta=1}^{m}a^{{\alpha\bar{\beta}}}a_{{\alpha\bar{\beta}}}=m,\qquad \sum_{i,j,k,l=1}^{m}a^{i\bar{l}}a^{k\bar{j}}a_{i\bar{j}}a_{k\bar{l}}=m,$$
hence we can apply \eqref{local tran of metric}, \eqref{local C^0 estimate after normal} and \eqref{upper C^0 close} to estimate at $x_0$
\[
\begin{split}
\left|T_0-m\right|
&=\left|\sum_{\alpha,\beta=1}^{m}h(t)^{\alpha\bar{\beta}}a_{\alpha\bar{\beta}}-\sum_{\alpha,\beta=1}^{m}a^{\alpha\bar{\beta}}a_{\alpha\bar{\beta}}\right|\\
&\leq\sum_{\alpha,\beta=1}^{m}\left|h(t)^{\alpha\bar{\beta}}-a^{\alpha\bar{\beta}}\right|\cdot\left|a_{\alpha\bar{\beta}}\right|\leq h_7(t),
\end{split}
\]
which gives \eqref{base trace convergence} at $x_0$, and similarly
\[\begin{split}
&\left|\|\om_B\|_{\om(t)}^2-m\right|=\left|\sum_{i,j,k,l=1}^{m}h(t)^{i\bar{l}}h(t)^{k\bar{j}}a_{i\bar{j}}a_{k\bar{l}}-\sum_{i,j,k,l=1}^{m}a^{i\bar{l}}a^{k\bar{j}}a_{i\bar{j}}a_{k\bar{l}}\right|\\
&\leq \sum_{i,j,k,l=1}^{m}\left(\left|h(t)^{i\bar{l}}-a^{i\bar{l}}\right|\left|h(t)^{k\bar{j}}a_{i\bar{j}}a_{k\bar{l}}\right|+\left|h(t)^{i\bar{l}}\right|\left|h(t)^{k\bar{j}}-a^{k\bar{j}}\right|\left|a_{i\bar{j}}a_{k\bar{l}}\right|\right)\\
&\leq h_8(t),\\
\end{split}
\]
which gives \eqref{base norm convergence} at $x_0$.  Since $x_0\in K$ was arbitrary chosen, we obtain \eqref{base trace convergence} and \eqref{base norm convergence}.  

When $S=\emptyset$, the above estimates hold on the whole manifold $X$ with $h(t)$ replaced by $Ce^{-\eta t}$ where $\eta,C>0$ are constants depending on $(X,\om_0)$ which may change from line to line.  This completes the proof.
\end{proof}

\section{The proof of theorem $\ref{main theorem}$}
In this section we prove Theorem $\ref{main theorem}$.  All the operators $\nabla,\Delta,\left\langle,\right\rangle$ are with respect to the evolving metric $\om(t)$.

We first need the following basic lemma to improve our decreasing function $h(t)$.
\begin{lemma}\label{good A(t)}
For any $h(t):[0,\infty)\to (0,\infty)$, a positive decreasing function which tends to zero as $t\to\infty$, there exists a smooth positive decreasing function $A(t):[0,\infty)\to (0,\infty)$ satisfying that: $h(t)\leq A(t),t\geq 0; A(t)\to 0$ as $t\to\infty$; and moreover
\begin{equation}\label{A(t) derivative bound}
0\leq -A'(t)\leq 100A(t),~on~[0,\infty).
\end{equation}
\end{lemma}
\begin{proof}
Choose $0<\ell_1<\ell_2<\cdots$ in the following way:  Let $\ell_1\gg 2$ such that $h(\ell_1)<\frac{1}{2}h(0)$.  Then choose $\ell_2\gg \ell_1+2$ such that $h(\ell_2)<\frac{1}{2^2}h(0)$.  Repeat this process, for each $k\geq 1$, we choose $\ell_{k+1}\gg\ell_k+2$ such that $h(\ell_{k+1})<\frac{1}{2^{k+1}}h(0)$.  First we define
$$A(t)\equiv h(0),~t\in[0,\ell_1];\qquad A(t)\equiv \frac{1}{2^k}h(0),~t\in[\ell_k+2,\ell_{k+1}],k\geq 1.$$
Hence for each $k\geq 1$, we have
\[
\left\{
     \begin{aligned}
     &h(t)\leq h(\ell_k)<\frac{1}{2^k}h(0),~t\in[\ell_k,\infty),\\
     &A(t)\equiv \frac{1}{2^{k-1}}h(0),~t\in[\ell_k-1,\ell_k],\\
     &A(t)\equiv \frac{1}{2^k}h(0),~t\in[\ell_k+2,\ell_k+3],\\
\end{aligned}
\right.
\]
hence we can define $A(t)$ on $[\ell_k,\ell_k+2]$ such that $A(t)$ is smooth and decreasing on $[\ell_k-1,\ell_k+3]$ (and hence smooth and decreasing on $(0,\infty)$) and moreover
$$0\leq -A'(t)\leq 200\cdot\frac{\frac{1}{2^{k-1}}h(0)-\frac{1}{2^k}h(0)}{(\ell_k+2)-\ell_k}=\frac{100}{2^k}h(0)\leq 100A(t),$$
for $t\in[\ell_k,\ell_k+2]$.  Outside such intervals, $A'(t)\equiv 0$, hence \eqref{A(t) derivative bound} is verified.  Also, on each $[\ell_k,\ell_{k+1}]$, we have
$$h(t)\leq\frac{1}{2^k}h(0)\leq A(t),$$
hence $h(t)\leq A(t)$ on $[0,\infty)$.  Finally, $A(t)\to 0$ as $t\to\infty$ is easy to see.  This finishes the proof.
\end{proof}

Next, we need to construct local cutoff function.
\begin{lemma}\label{cutoff function}
Recall that $K=f^{-1}(K'),U=f^{-1}(U'),K'\subset\subset U'\subset\subset B\backslash S'$.  Then there exists a smooth cutoff function $\rho$ with $supp(\rho)\subset U$,$\rho>0$ on $U$, $0\leq\rho\leq 1$, $\rho\equiv 1$ on $K$, satisfying
\begin{equation}\label{cutoff bound}
|\nabla{\rho}|_{\om(t)}^2+|\Delta_{\om(t)}\rho|\leq C,
\end{equation}
on $U\times [0,\infty)$ for some constant $C$ depending on the domain $K$.
\end{lemma}
\begin{proof}
We first choose cutoff function $\rho_0$ on $B$ such that $supp(\rho_0)\subset U'$, $\rho_0>0$ on $U'$, $0\leq\rho_0\leq 1$, $\rho_0\equiv 1$ on $K'$ and moreover
$$\sqrt{-1}\partial\rho_0\wedge\dbar\rho_0\leq C\om_B,\qquad -C\om_B\leq\ddbar\rho_0\leq C\om_B,$$
on $U'$.  Then we set $\rho=f^*\rho_0$.  Using Lemma \ref{basic} and Lemma \ref{C^0 estimate}, we have under local product coordinates
\[\begin{split}
&|\nabla\rho|_{\om(t)}^2=\sum_{i,j=1}^{m}g(t)^{i\bar{j}}\partial_i\rho\partial_{\bar{j}}\rho\leq C,\\
&|\Delta_{\om(t)}\rho|=|\tr{\om(t)}{\ddbar\rho}|\leq C\tr{\om(t)}{\om_B}\leq C,\\
\end{split}
\]
with some constant $C$ depending on the domain $K$.  This finishes the proof.
\end{proof}

In the following, we always use $h(t),A(t),B(t)$ to denote positive decreasing functions which tend to zero as $t\to\infty$, and moreover $A(t),B(t)$ satisfy condition \eqref{A(t) derivative bound}.

Set $u=\dot{\vp}+\vp-v$ on $X\backslash S$. Recall that $T_0=\tr{\om(t)}{\om_B}$, and
$$\left(\ddt-\Delta\right)u=T_0-m.$$
Then along the flow we have on $X\backslash S$
\begin{equation}\label{evolution of 1st derivative of u}
\left(\ddt-\Delta\right)|\nabla u|^2=|\nabla u|^2-|\nabla\nabla u|^2-|\nabla\bar{\nabla}u|^2+2Re(\nabla T_0\cdot\bar{\nabla}u).
\end{equation}
\begin{equation}\label{evolution of laplacian u}
\left(\ddt-\Delta\right)\Delta u=\Delta u+\left\langle\Ric,\ddbar u\right\rangle+\Delta T_0.
\end{equation}

We first estimate $|\nabla u|^2$.
\begin{proposition}\label{first derivative convergence}
There exists $A(t)$ which depends on the domain $K$ such that
\begin{equation}\label{first derivative convergence2}
|\nabla u|^2\leq A(t),~on~K\times[0,\infty).
\end{equation}
In particular, when $S=\emptyset$, then we have
$$|\nabla u|^2\leq Ce^{-\eta t},~on~X\times[0,\infty).$$
for some constants $\eta,C>0$ depending on $(X,\om_0)$. 
\end{proposition}
\begin{proof}
Apply Lemma \ref{basic} and Proposition \ref{base convergence}, we can find some $h(t)$ depending on $K$ such that $|u|+|T_0-m|\leq h(t)$ on $\bar{U}\times[0,\infty)$.  Then we choose $A(t)$ according to Lemma \ref{good A(t)}, such that $2h(t)\leq A(t),~t\geq 0$.  Hence
\begin{equation}\label{all C^0 bound1}
A(t)-u\in\left[\frac{1}{2}A(t),2A(t)\right],~|T_0-m|\leq A(t),~on~\bar{U}\times[0,\infty).
\end{equation}
So we can compute on $U\times[0,\infty)$ (see \cite{ST3})
\[
\begin{split}
&\left(\ddt-\Delta\right)\left(\frac{|\nabla u|^2}{A(t)-u}\right)\\
=
&\frac{|\nabla u|^2-\left(|\nabla\nabla u|^2+|\nabla\bar{\nabla }u|^2\right)+2Re(\nabla T_0\cdot\bar{\nabla}u)}{A(t)-u}-2\epsilon\frac{Re\left[\nabla|\nabla u|^2\cdot\bar{\nabla}u\right]}{(A(t)-u)^2}\\
&-2\epsilon\frac{|\nabla u|^4}{(A(t)-u)^3}-\frac{2(1-\epsilon)}{A(t)-u}Re\left[\nabla\left(\frac{|\nabla u|^2}{A(t)-u}\right)\cdot\bar{\nabla }u\right]\\
&+(T_0-m)\frac{|\nabla u|^2}{(A(t)-u)^2}-A'(t)\frac{|\nabla u|^2}{(A(t)-u)^2},\\
\end{split}
\]
for any $\epsilon\in\mathbb{R}$.  Now for some $k\in\mathbb{R}$ to be fixed, we can compute
$$ \left(\ddt-\Delta\right)\frac{1}{(A(t)-u)^k}=-\frac{kA'(t)}{(A(t)-u)^{1+k}}-\frac{k(k+1)|\nabla u|^2}{(A(t)-u)^{2+k}}+\frac{k(T_0-m)}{(A(t)-u)^{1+k}},$$
and
\[
\begin{split}
&\frac{2}{(A(t)-u)^{1+k}}Re\left[\nabla\left(\frac{|\nabla u|^2}{A(t)-u}\right)\cdot\bar{\nabla}u\right]\\
&=\frac{2}{A(t)-u}Re\left[\nabla\left(\frac{|\nabla u|^2}{(A(t)-u)^{1+k}}\right)\cdot\bar{\nabla}u\right]-2k\frac{|\nabla u|^4}{(A(t)-u)^{3+k}},\\
\end{split}
\]
and hence on $U\times[0,\infty)$
\begin{equation}\label{evolution1}
\begin{split}
&\left(\ddt-\Delta\right)\left(\frac{|\nabla u|^2}{(A(t)-u)^{1+k}}\right)\\
&=\frac{\left(\ddt-\Delta\right)\left(\frac{|\nabla u|^2}{A(t)-u}\right)}{(A(t)-u)^{k}}+\frac{|\nabla u|^2}{A(t)-u}\left(\ddt-\Delta\right)\frac{1}{(A(t)-u)^k}\\
&-2Re\left[\nabla\left(\frac{|\nabla u|^2}{A(t)-u}\right)\cdot\bar{\nabla}\left(\frac{1}{(A(t)-u)^k}\right)\right]\\
&=\frac{|\nabla u|^2-(|\nabla\nabla u|^2+|\nabla\bar{\nabla}u|^2)+2Re(\nabla T_0\cdot\bar{\nabla}u)}{(A(t)-u)^{1+k}}-2\epsilon\frac{Re\left[\nabla|\nabla u|^2\cdot\bar{\nabla}u\right]}{(A(t)-u)^{2+k}}\\
&-[2\epsilon+k(k+1)-2k(1-\epsilon+k)]\frac{|\nabla u|^4}{(A(t)-u)^{3+k}}+(1+k)\frac{(T_0-m)|\nabla u|^2}{(A(t)-u)^{2+k}}\\
&-\frac{2(1-\epsilon+k)}{A(t)-u}Re\left[\nabla\left(\frac{|\nabla u|^2}{(A(t)-u)^{1+k}}\right)\cdot\bar{\nabla}u\right]-(1+k)\frac{A'(t)|\nabla u|^2}{(A(t)-u)^{2+k}}.\\
\end{split}
\end{equation}
Now we set, in this proof, $k=-\frac{1}{3}$, $\epsilon=\frac{2}{3}$, then
$$1-\epsilon+k=0,~2\epsilon+k(1+k)-2k(1-\epsilon+k)=\frac{10}{9},$$
hence \eqref{evolution1} becomes
\begin{equation}\label{evolution2}
\begin{split}
&\left(\ddt-\Delta\right)\left(\frac{|\nabla u|^2}{(A(t)-u)^{1+k}}\right)\\
=&\frac{|\nabla u|^2-(|\nabla\nabla u|^2+|\nabla\bar{\nabla}u|^2)+2Re(\nabla T_0\cdot\bar{\nabla}u)}{(A(t)-u)^{1+k}}-\frac{4}{3}\cdot\frac{Re[\nabla|\nabla u|^2\cdot\bar{\nabla}u]}{(A(t)-u)^{2+k}}\\
&-\frac{10}{9}\cdot\frac{|\nabla u|^4}{(A(t)-u)^{3+k}}+\frac{2}{3}\cdot\frac{(T_0-m)|\nabla u|^2}{(A(t)-u)^{2+k}}-\frac{2}{3}\cdot\frac{A'(t)|\nabla u|^2}{(A(t)-u)^{2+k}}.\\
\end{split}
\end{equation}
We come to estimate each term.  First, if we choose normal coordinates around a point in $U$, then we have
$$\left|\nabla|\nabla u|^2\cdot\bar{\nabla}u\right|=\left|u_i(u_ju_{\bar{j}})_{\bar{i}}\right|=|u_iu_ju_{\bar{j}\bar{i}}+u_iu_{\bar{j}}u_{j\bar{i}}|\leq |\nabla u|^2\left(|\nabla\nabla u|+|\nabla\bar{\nabla}u|\right),$$
hence
\[
\begin{split}
&\left|\frac{4}{3}\cdot\frac{Re\left[\nabla|\nabla u|^2\cdot\bar{\nabla}u\right]}{(A(t)-u)^{2+k}}\right|\\
&\leq 2\left(\frac{2}{3}\cdot\frac{|\nabla u|^2}{(A(t)-u)^{\frac{3+k}{2}}}\right)\left(\frac{|\nabla\nabla u|}{(A(t)-u)^{\frac{1+k}{2}}}+\frac{|\nabla\bar{\nabla} u|}{(A(t)-u)^{\frac{1+k}{2}}}\right)\\
&\leq \frac{|\nabla\nabla u|^2+|\nabla\bar{\nabla}u|^2}{(A(t)-u)^{1+k}}+\frac{8}{9}\cdot\frac{|\nabla u|^4}{(A(t)-u)^{3+k}},\\
\end{split}
\]
Next, since $k=-\frac{1}{3}$, we have $4+4k=\frac{8}{3}=3+k$, hence
\[
\begin{split}
\frac{\left|2Re(\nabla T_0\cdot\bar{\nabla}u)\right|}{(A(t)-u)^{1+k}}
&\leq \frac{2|\nabla T_0|\cdot|\nabla u|}{(A(t)-u)^{1+k}}\leq |\nabla T_0|^2+\frac{|\nabla u|^2}{(A(t)-u)^{2+2k}}\\
&\leq |\nabla T_0|^2+\frac{1}{100}\cdot\frac{|\nabla u|^4}{(A(t)-u)^{4+4k}}+100\\
&=|\nabla T_0|^2+\frac{1}{100}\cdot\frac{|\nabla u|^4}{(A(t)-u)^{3+k}}+100,\\
\end{split}
\]
The above two terms are the main terms that we need to be careful about, rest three terms are easy to control by using \eqref{A(t) derivative bound} and \eqref{all C^0 bound1} (remember that $k=-\frac{1}{3}$):
\[
\begin{split}
\frac{|\nabla u|^2}{(A(t)-u)^{1+k}}
&\leq \frac{1}{100}\cdot \frac{|\nabla u|^4}{(A(t)-u)^{3+k}}+100(A(t)-u)^{1-k}\\
&\leq \frac{1}{100}\cdot \frac{|\nabla u|^4}{(A(t)-u)^{3+k}}+C,\\
\end{split}
\]
$$\left|\frac{2}{3}\cdot\frac{(T_0-m)|\nabla u|^2}{(A(t)-u)^{2+k}}\right|\leq \frac{A(t)|\nabla u|^2}{\frac{1}{2}A(t)\cdot(A(t)-u)^{1+k}}\leq \frac{1}{100}\cdot \frac{|\nabla u|^4}{(A(t)-u)^{3+k}}+C,$$
$$\left|\frac{2}{3}\cdot\frac{A'(t)|\nabla u|^2}{(A(t)-u)^{2+k}}\right|\leq \frac{100A(t)|\nabla u|^2}{\frac{1}{2}A(t)\cdot(A(t)-u)^{1+k}}\leq \frac{1}{100}\cdot \frac{|\nabla u|^4}{(A(t)-u)^{3+k}}+C,$$
Hence we obtain on $U\times[0,\infty)$
\begin{equation}\label{evolution3}
\left(\ddt-\Delta\right)\left(\frac{|\nabla u|^2}{(A(t)-u)^{1+k}}\right)\leq -\frac{1}{10}\cdot\frac{|\nabla u|^4}{(A(t)-u)^{3+k}}+|\nabla T_0|^2+C.
\end{equation}
Next, from Lemma $\ref{Schwarz Lemma}$ we have on $U\times[0,\infty)$
\begin{equation}\label{evolution4}
\left(\ddt-\Delta\right)T_0^2=2T_0\left(\ddt-\Delta\right)T_0-2|\nabla T_0|^2\leq C-2|\nabla T_0|^2.
\end{equation}
Hence if we set
$$Q=\frac{|\nabla u|^2}{(A(t)-u)^{1+k}}+T_0^2,$$
then we obtain from \eqref{evolution3} and \eqref{evolution4} that on $U\times[0,\infty)$
\begin{equation}\label{evolution5}
\left(\ddt-\Delta\right)Q\leq -\frac{1}{10}\cdot\frac{|\nabla u|^4}{(A(t)-u)^{3+k}}+C.
\end{equation}

Now we choose cutoff function $\rho$ according to Lemma $\ref{cutoff function}$.  Then we can compute
$$\left(\ddt-\Delta\right)\left(\rho^4Q\right)=\rho^4\left(\ddt-\Delta\right)Q-Q\Delta\rho^4-2Re\left[\nabla Q\cdot \bar{\nabla}\rho^4\right],$$
For the second term, we use \eqref{cutoff bound} to estimate on $U\times[0,\infty)$
\[
\begin{split}
-Q\Delta\rho^4
&\leq C\rho^2Q=C\rho^2\frac{|\nabla u|^2}{(A(t)-u)^{1+k}}+C\rho^2T_0^2\\
&\leq \frac{1}{100}\cdot \frac{\rho^4|\nabla u|^4}{(A(t)-u)^{3+k}}+C(A(t)-u)^{1-k}+C\rho^2T_0^2\\
&\leq \frac{1}{100}\cdot \frac{\rho^4|\nabla u|^4}{(A(t)-u)^{3+k}}+C,\\
\end{split}
\]
For the third term, since $\rho>0$ on $U$, we have on $U\times[0,\infty)$
\[
\begin{split}
-2Re\left[\nabla Q\cdot \bar{\nabla}\rho^4\right]
&=-2Re\left[\frac{\nabla(\rho^4Q)-4\rho^3Q\nabla\rho}{\rho^4}\cdot 4\rho^3\bar{\nabla}\rho\right]\\
&=-\frac{8}{\rho}Re\left[\nabla(\rho^4Q)\cdot\bar{\nabla}\rho\right]+32\rho^2|\nabla\rho|^2Q\\
&\leq-\frac{8}{\rho}Re\left[\nabla(\rho^4Q)\cdot\bar{\nabla}\rho\right]+\frac{1}{100}\cdot \frac{\rho^4|\nabla u|^4}{(A(t)-u)^{3+k}}+C,\\
\end{split}
\]
hence combining \eqref{evolution5} we obtain on $U\times[0,\infty)$
\begin{equation}\label{evolution6}
\left(\ddt-\Delta\right)\left(\rho^4Q\right)\leq -\frac{1}{20}\cdot\frac{\rho^4|\nabla u|^4}{(A(t)-u)^{3+k}}-\frac{8}{\rho}Re\left[\nabla(\rho^4Q)\cdot\bar{\nabla}\rho\right]+C.
\end{equation}

Now, assume $\rho^4Q$ achieves its maximum at $(x_0,t_0)$ with $t_0>0$, then $x_0\notin\partial U$, hence $x_0\in U$ and then $\rho(x_0)>0$ and we have
$$-\frac{8}{\rho}Re\left[\nabla(\rho^4Q)\cdot\bar{\nabla}\rho\right](x_0,t_0)=0,$$
Then we apply maximum principle to \eqref{evolution6} to obtain
$$0\leq\left(\ddt-\Delta\right)\left(\rho^4Q\right)(x_0,t_0)\leq -\frac{1}{20}\cdot\frac{\rho^4|\nabla u|^4}{(A(t)-u)^{3+k}}(x_0,t_0)+C,$$
which gives at $(x_0,t_0)$
$$\frac{\rho^4|\nabla u|^2}{(A(t)-u)^{1+k}}\leq \frac{\rho^4|\nabla u|^4}{(A(t)-u)^{3+k}}+\rho^4(A(t)-u)^{1-k}\leq C.$$
But $\rho^4T_0^2\leq C$ on $U\times[0,\infty)$, we conclude that $\rho^4Q\leq C$ on $U\times[0,\infty)$, which gives us that
$$\frac{|\nabla u|^2}{(A(t)-u)^{\frac{2}{3}}}\leq C,~on~K\times[0,\infty),$$
where $C$ is some constant depending on the domain $K$.  Using \eqref{all C^0 bound1} again, we obtain \eqref{first derivative convergence2} with some larger $A(t)$.

Finally, when $S=\emptyset$, the above arguments are still true on $X\times[0,\infty)$ with all $h(t)$ and $A(t)$ replaced by $Ce^{-\eta t}$ with $\eta,C>0$ are constants depending on $(X,\om_0)$ which may change from line to line, since its easy to see that
$$\left|\frac{(Ce^{-\eta t})'}{Ce^{-\eta t}}\right|=\eta\leq 1.$$
(this is the motivation of Lemma \ref{good A(t)}).  This completes the proof.
\end{proof}

Now we come to estimate $|\Delta u|$ locally.  We have the following proposition.
\begin{proposition}\label{Laplacian u convergence}
There exists $A(t)$ depending on the domain $K$ such that
\begin{equation}\label{Laplacian u convergence2}
|\Delta u|\leq A(t),~on~K\times[0,\infty).
\end{equation}
In particular, when $S=\emptyset$, then we have
$$|\Delta u|\leq Ce^{-\eta t},~on~X\times[0,\infty).$$
for some constants $\eta,C>0$ depending on $(X,\om_0)$. 
\end{proposition}
\begin{proof}
Applying Lemma \ref{basic}, Proposition $\ref{base convergence}$ and Proposition $\ref{first derivative convergence}$, we can find $h(t)$ such that
$$|T_0-m|+\left|\|\om_B\|_{\om(t)}^2-m\right|+|u|+|\nabla u|^2\leq \frac{1}{2}h(t),~on~\bar{U}\times[0,\infty),$$
Then we apply Lemma \ref{good A(t)} to find some $B(t)$ which then depends on the domain $K$ such that $h(t)\leq B(t)$ for $t\geq 0$ and $B(t)$ satisfies  \eqref{A(t) derivative bound}.  WLOG, we may assume $B(t)\leq 1$ for $t\geq 0$, since otherwise we can consider $t\in[T,\infty)$ for some $T$ large.  Now we set $A(t)=B(t)^{\frac{1}{2}}\geq B(t)$, then we have $A(t)\leq 1$ for $t\geq 0$ and moreover
\begin{equation}\label{all C^0 bound2}
\left\{
\begin{aligned}
      &|T_0-m|+\left|\|\om_B\|_{\om(t)}^2-m\right|+|u|\leq\frac{1}{2}A(t),\\
      &|\nabla u|^2\leq A(t)^2,\\
\end{aligned}
\right.
\end{equation}
on $\bar{U}\times[0,\infty)$.  Still, we have
$$0\leq -A'(t)=-\frac{1}{2}\cdot\frac{B'(t)}{B(t)^{\frac{1}{2}}}\leq \frac{1}{2}\cdot\frac{100B(t)}{B(t)^{\frac{1}{2}}}\leq 100B(t)^{\frac{1}{2}}=100A(t),$$
which means that $A(t)$ still satisfies \eqref{A(t) derivative bound}.

Now we  use \eqref{evolution of laplacian u} to compute
\begin{equation}\label{evolution7}
\begin{split}
&\left(\ddt-\Delta\right)\left(\frac{-\Delta u}{A(t)-u}\right)\\
=&\frac{-\Delta u+|\nabla\bar{\nabla}u|^2+\left\langle\ddbar u,\om_B\right\rangle-\Delta T_0}{A(t)-u}-\frac{(T_0-m)\Delta u}{(A(t)-u)^2}\\
&-\frac{2}{A(t)-u}Re\left[\nabla\left(\frac{-\Delta u}{A(t)-u}\right)\cdot\bar{\nabla}u\right]+\frac{A'(t)\Delta u}{(A(t)-u)^2},\\
\end{split}
\end{equation}
which is always meaningful on $U\times[0,\infty)$ thanks to \eqref{all C^0 bound2}, where we have used the fact that on $X\backslash S\times[0,\infty)$
\begin{equation}\label{relation of Ric and u}
\begin{split}
\Ric(\om)
&=-\ddbar (\dot{\vp}+\vp)-\chi\\
&=-\ddbar (\dot{\vp}+\vp-v)-(\chi+\ddbar v)\\
&=-\ddbar u-\om_B.\\
\end{split}
\end{equation}
Also, we have
\[
\begin{split}
\left(\ddt-\Delta\right)\left(\frac{T_0-m}{A(t)-u}\right)
&=\frac{\ddt T_0-\Delta T_0}{A(t)-u}-\frac{A'(t)(T_0-m)}{(A(t)-u)^2}\\
&+\frac{(T_0-m)^2}{(A(t)-u)^2}-\frac{2}{A(t)-u}Re\left[\nabla\left(\frac{T_0-m}{A(t)-u}\right)\cdot\bar{\nabla}u\right].\\
\end{split}
\]
Hence if we set
$$K=\frac{-\Delta u-(T_0-m)}{A(t)-u},$$
then we have
\[
\begin{split}
&\left(\ddt-\Delta\right)K\\
=&\frac{-\Delta u+|\nabla\bar{\nabla}u|^2+\left\langle\ddbar u,\om_B\right\rangle}{A(t)-u}-\frac{2}{A(t)-u}Re\left[\nabla K\cdot\bar{\nabla}u\right]\\
&-\frac{(T_0-m)\Delta u}{(A(t)-u)^2}+\frac{A'(t)\Delta u}{(A(t)-u)^2}-\frac{\ddt T_0}{A(t)-u}+\frac{A'(t)(T_0-m)}{(A(t)-u)^2}-\frac{(T_0-m)^2}{(A(t)-u)^2},\\
\end{split}
\]
But using \eqref{relation of Ric and u} we have on $X\backslash S$
\[
\begin{split}
\ddt T_0
&=\ddt\left(g(t)^{i\bar{j}}(g_B)_{i\bar{j}}\right)=\left(R^{\bar{j}i}+g(t)^{i\bar{j}}\right)(g_B)_{i\bar{j}}\\
&=\left\langle\Ric,\om_B\right\rangle+T_0=-\left\langle\ddbar u,\om_B\right\rangle+T_0-\|\om_B\|_{\om(t)}^2,\\
\end{split}
\]
hence we obtain
\begin{equation}\label{evolution8}
\begin{split}
&\left(\ddt-\Delta\right)K\\=
&\frac{-\Delta u+|\nabla\bar{\nabla}u|^2+2\left\langle\ddbar u,\om_B\right\rangle-T_0+\|\om_B\|_{\om(t)}^2}{A(t)-u}\\
&-\frac{2}{A(t)-u}Re\left[\nabla K\cdot\bar{\nabla}u\right]-\frac{(T_0-m)\Delta u}{(A(t)-u)^2}+\frac{A'(t)\Delta u}{(A(t)-u)^2}\\
&+\frac{A'(t)(T_0-m)}{(A(t)-u)^2}-\frac{(T_0-m)^2}{(A(t)-u)^2},\\
\end{split}
\end{equation}
Next, set $\epsilon=k=1$ in \eqref{evolution1}, we have (note that $A(t)$ is changed here)
\begin{equation}\label{evolution9}
\begin{split}
&\left(\ddt-\Delta\right)\left(\frac{|\nabla u|^2}{(A(t)-u)^{2}}\right)\\
=&\frac{|\nabla u|^2-(|\nabla\nabla u|^2+|\nabla\bar{\nabla}u|^2)+2Re(\nabla T_0\cdot\bar{\nabla}u)}{(A(t)-u)^{2}}\\
&-2\frac{Re\left[\nabla|\nabla u|^2\cdot\bar{\nabla}u\right]}{(A(t)-u)^{3}}-2\frac{|\nabla u|^4}{(A(t)-u)^{4}}+\frac{2(T_0-m)|\nabla u|^2}{(A(t)-u)^{3}}\\
&-\frac{2A'(t)|\nabla u|^2}{(A(t)-u)^{3}}-\frac{2}{A(t)-u}Re\left[\nabla\left(\frac{|\nabla u|^2}{(A(t)-u)^{2}}\right)\cdot\bar{\nabla}u\right].\\
\end{split}
\end{equation}
Now set $H=\frac{|\nabla u|^2}{(A(t)-u)^{2}}$ and then set
$$Q_{\pm}=\pm K+100H+T_0^2,$$
then combining \eqref{evolution8} and \eqref{evolution9} we obtain that on $U\times[0,\infty)$
\begin{equation}\label{evolution10}
\begin{split}
&\left(\ddt-\Delta\right)Q_\pm\\=
&\pm\left\{\frac{-\Delta u+|\nabla\bar{\nabla}u|^2+2\left\langle\ddbar u,\om_B\right\rangle-T_0+\|\om_B\|_{\om(t)}^2}{A(t)-u}\right.\\
&\left.-\frac{(T_0-m)\Delta u}{(A(t)-u)^2}+\frac{A'(t)\Delta u}{(A(t)-u)^2}+\frac{A'(t)(T_0-m)}{(A(t)-u)^2}-\frac{(T_0-m)^2}{(A(t)-u)^2}\right\}\\
&+100\left\{\frac{|\nabla u|^2-(|\nabla\nabla u|^2+|\nabla\bar{\nabla}u|^2)+2Re(\nabla T_0\cdot\bar{\nabla}u)}{(A(t)-u)^{2}}\right.\\
&\left.-\frac{2Re\left[\nabla|\nabla u|^2\cdot\bar{\nabla}u\right]}{(A(t)-u)^{3}}-\frac{2|\nabla u|^4}{(A(t)-u)^{4}}+\frac{2(T_0-m)|\nabla u|^2}{(A(t)-u)^{3}}-\frac{2A'(t)|\nabla u|^2}{(A(t)-u)^{3}}\right\}\\
&-\frac{2}{A(t)-u}Re\left[\nabla Q_\pm\cdot\bar{\nabla}u\right]+\frac{4T_0}{A(t)-u}Re(\nabla T_0\cdot\bar{\nabla}u)\\
&+2T_0\left(\ddt-\Delta\right)T_0-2|\nabla T_0|^2.\\
\end{split}
\end{equation}
With the help of \eqref{all C^0 bound2}, we only need two ``good terms''
$$-100\frac{|\nabla\nabla u|^2+|\nabla\bar{\nabla}u|^2}{(A(t)-u)^2},\qquad -2|\nabla T_0|^2,$$
to control all other terms except the term which involves $\nabla Q_\pm$.  Indeed, we have on $U\times[0,\infty)$
$$\left|\frac{\Delta u}{A(t)-u}\right|\leq \frac{C|\nabla\bar{\nabla}u|}{A(t)-u}\leq \frac{|\nabla\bar{\nabla}u|^2}{(A(t)-u)^2}+C;$$
$$\frac{|\nabla\bar{\nabla}u|^2}{(A(t)-u)}\leq \frac{2A(t)|\nabla\bar{\nabla}u|^2}{(A(t)-u)^2}\leq 2\frac{|\nabla\bar{\nabla}u|^2}{(A(t)-u)^2};$$
$$\left|\frac{2\left\langle\ddbar u,\om_B\right\rangle}{A(t)-u}\right|\leq \frac{2|\nabla\bar{\nabla}u|\cdot\|\om_B\|_{\om(t)}}{A(t)-u}\leq \frac{|\nabla\bar{\nabla}u|^2}{(A(t)-u)^2}+C;$$
$$\left|\frac{T_0-\|\om_B\|_{\om(t)}^2}{A(t)-u}\right|\leq \frac{|T_0-m|+\left|\|\om_B\|_{\om(t)}^2-m\right|}{A(t)-u}\leq \frac{\frac{1}{2}A(t)+\frac{1}{2}A(t)}{\frac{1}{2}A(t)}=2;$$
$$\left|\frac{(T_0-m)\Delta u}{(A(t)-u)^2}\right|\leq \frac{A(t)\cdot C|\nabla\bar{\nabla}u|}{(A(t)-u)^2}\leq \frac{|\nabla\bar{\nabla}u|^2}{(A(t)-u)^2}+C;$$
$$\left|\frac{A'(t)\Delta u}{(A(t)-u)^2}\right|\leq \frac{100A(t)\cdot C|\nabla\bar{\nabla}u|}{(A(t)-u)^2}\leq \frac{|\nabla\bar{\nabla}u|^2}{(A(t)-u)^2}+C;$$
$$\left|\frac{A'(t)(T_0-m)}{(A(t)-u)^2}-\frac{(T_0-m)^2}{(A(t)-u)^2}\right|\leq C;$$
$$100\frac{|\nabla u|^2}{(A(t)-u)^2}\leq 100\frac{A(t)^2}{\frac{1}{4}A(t)^2}\leq C;$$
$$\left|\frac{2Re(\nabla T_0\cdot\bar{\nabla}u)}{A(t)-u}\right|\leq \frac{2|\nabla T_0|\cdot A(t)}{\frac{1}{2}A(t)}\leq |\nabla T_0|^2+C;$$
and moreover
\[
\begin{split}
&100\left\{-\frac{2Re\left[\nabla|\nabla u|^2\cdot\bar{\nabla}u\right]}{(A(t)-u)^{3}}+\frac{2(T_0-m)|\nabla u|^2}{(A(t)-u)^{3}}-\frac{2A'(t)|\nabla u|^2}{(A(t)-u)^{3}}\right\}\\
&\leq C\left\{\frac{A(t)^2\left(|\nabla\nabla u|+|\nabla\bar{\nabla}u|\right)}{\frac{1}{4}A(t)^2\cdot(A(t)-u)}+\frac{A(t)^3}{A(t)^3}\right\}\\
&\leq \frac{|\nabla\nabla u|^2+|\nabla\bar{\nabla}u|^2}{(A(t)-u)^2}+C;\\
\end{split}
\]
$$\frac{4T_0}{A(t)-u}Re(\nabla T_0\cdot\bar{\nabla}u)+2T_0\left(\ddt-\Delta\right)T_0\leq |\nabla T_0|^2+C;$$
Hence we conclude that on $U\times[0,\infty)$
$$\left(\ddt-\Delta\right)Q_\pm\leq -5\frac{|\nabla\bar{\nabla}u|^2}{(A(t)-u)^2}-\frac{2}{A(t)-u}Re\left[\nabla Q_\pm\cdot\bar{\nabla}u\right]+C.$$

Now as before, we choose cutoff function $\rho$ according to Lemma \ref{cutoff function}, then we have that on $U\times[0,\infty)$
\[
\begin{split}
&\left(\ddt-\Delta\right)(\rho^4Q_\pm)\\
&\leq -5\frac{\rho^4|\nabla\bar{\nabla}u|^2}{(A(t)-u)^2}-\frac{2\rho^4Re\left[\nabla Q_\pm\cdot\bar{\nabla}u\right]}{A(t)-u}+C-Q_\pm\Delta \rho^4-2Re\left[\nabla Q_\pm\cdot\bar{\nabla}\rho^4\right],\\
\end{split}
\]
For the forth term, we have
\[
\begin{split}
-Q_\pm\Delta \rho^4
&\leq C\rho^2|Q_\pm|\\
&\leq C\rho^2\left\{\frac{|\nabla\bar{\nabla}u|}{A(t)-u}+\frac{|T_0-m|}{A(t)-u}+100\frac{|\nabla u|^2}{(A(t)-u)^2}+T_0^2\right\}\\
&\leq \frac{\rho^4|\nabla\bar{\nabla}u|^2}{(A(t)-u)^2}+C;\\
\end{split}
\]
for the second term, we have
\[
\begin{split}
&-\frac{2\rho^4Re\left[\nabla Q_\pm\cdot\bar{\nabla}u\right]}{A(t)-u}\\
&=-\frac{2}{A(t)-u}Re\left[\nabla(\rho^4Q_\pm)\cdot\bar{\nabla}u\right]+\frac{8\rho^3Q_\pm}{A(t)-u}Re\left[\nabla\rho\cdot\bar{\nabla}u\right]\\
&\leq -\frac{2}{A(t)-u}Re\left[\nabla(\rho^4Q_\pm)\cdot\bar{\nabla}u\right]+\frac{C\rho^3|Q_\pm|A(t)^2}{A(t)}\\
&\leq -\frac{2}{A(t)-u}Re\left[\nabla(\rho^4Q_\pm)\cdot\bar{\nabla}u\right]+\frac{\rho^4|\nabla\bar{\nabla}u|^2}{(A(t)-u)^2}+C;\\
\end{split}
\]
and similarly for the last term
$$-2Re\left[\nabla Q_\pm\cdot\bar{\nabla}\rho^4\right]\leq -\frac{8}{\rho}Re\left[\nabla(\rho^4Q_\pm)\cdot\bar{\nabla}\rho\right]+\frac{\rho^4|\nabla\bar{\nabla}u|^2}{(A(t)-u)^2}+C.$$
Hence we finally conclude on $U\times[0,\infty)$
\begin{equation}\label{evolution11}
\begin{split}
&\left(\ddt-\Delta\right)(\rho^4Q_\pm)\\
&\leq -\frac{\rho^4|\nabla\bar{\nabla}u|^2}{(A(t)-u)^2}-\frac{2Re\left[\nabla(\rho^4Q_\pm)\cdot\bar{\nabla}u\right]}{A(t)-u}-\frac{8}{\rho}Re\left[\nabla(\rho^4Q_\pm)\cdot\bar{\nabla}\rho\right]+C.\\
\end{split}
\end{equation}
Now we assume $\rho^4Q_+$ achieves its maximum at $(x_0,t_0)$ with $t_0>0$, then if $x_0\in\partial U$, we are done.  Hence we can assume that $x_0\in U$ and then $\rho(x_0)>0$ and hence
$$-\frac{2Re\left[\nabla(\rho^4Q_+)\cdot\bar{\nabla}u\right]}{A(t)-u}(x_0,t_0)-\frac{8}{\rho}Re\left[\nabla(\rho^4Q_+)\cdot\bar{\nabla}\rho\right](x_0,t_0)=0,$$
then we obtain from maximum principle and \eqref{evolution11} that
$$0\leq\left(\ddt-\Delta\right)(\rho^4Q_+)(x_0,t_0)\leq -\frac{\rho^4|\nabla\bar{\nabla}u|^2}{(A(t)-u)^2}(x_0,t_0)+C,$$
and hence on $U\times[0,\infty)$
\[
\begin{split}
\rho^4Q_+
&\leq \rho^4Q_+(x_0,t_0)\\
&=\rho^4\left\{-\frac{\Delta u+(T_0-m)}{A(t)-u}+100\frac{|\nabla u|^2}{(A(t)-u)^2}+T_0^2\right\}(x_0,t_0)\\
&\leq \frac{\rho^4|\nabla\bar{\nabla}u|}{A(t)-u}(x_0,t_0)+C\leq C,\\
\end{split}
\]
which gives
$$\frac{-\Delta u}{A(t)-u}\leq C,~on~K\times[0,\infty).$$
Similarly, consider $Q_-$ instead gives us
$$\frac{\Delta u}{A(t)-u}\leq C,~on~K\times[0,\infty).$$
and hence we conclude
$$|\Delta u|\leq A(t),~on~K\times[0,\infty),$$
for some larger $A(t)$.  Hence we obtain \eqref{Laplacian u convergence2}.

Finally, when $S=\emptyset$, the above arguments are still true on $X\times[0,\infty)$ with all $h(t)$ and $A(t)$ replaced by $Ce^{-\eta t}$ with $\eta,C>0$ are constants depending on $(X,\om_0)$ which may change from line to line.  This completes the proof.
\end{proof}

Now we can prove Theorem $\ref{main theorem}$.
\begin{proof}[Proof of Theorem $\ref{main theorem}$]
From \eqref{relation of Ric and u}, we have that on $X\backslash S\times[0,\infty)$
$$R=-T_0-\Delta u,$$
hence Proposition \ref{base convergence} and Proposition \ref{Laplacian u convergence} give us some $h(t)$ depending on the domain $K$ such that
$$|R+m|\leq |T_0-m|+|\Delta u|\leq h(t),~on~K\times[0,\infty).$$
In particular, if $S=\emptyset$, then
$$|R+m|\leq |T_0-m|+|\Delta u|\leq Ce^{-\eta t},~on~X\times[0,\infty),$$
where $\eta,C>0$ are constants depending on $(X,\om_0)$.  This completes the proof.
\end{proof}

\end{document}